\documentclass[12pt]{article}
\usepackage{amsthm}
\usepackage{amsfonts}
\usepackage[T1]{fontenc}
\usepackage{epsfig}

\usepackage[lflt]{floatflt}
\usepackage{epsfig}
\usepackage{graphicx}
\usepackage{amsmath}
\usepackage{amssymb}
\usepackage{color}
\usepackage{subfigure}
\usepackage[english]{babel}
\usepackage{empheq}
\usepackage{caption}

\usepackage{makeidx}
\usepackage{latexsym}
\usepackage{enumerate}
\newtheorem{note}{Note}
\newtheorem{theo}{Theorem}
\newtheorem{defi}{Definition}
\newtheorem{prop}{Proposition}
\newtheorem{lemma}{Lemma}

\newtheorem{remark}{Remark}
\newtheorem{example}{Example}

\newcommand{\bbN}{{\mathbb N}}
\newcommand{\bbR}{{\mathbb R}}

\newcommand{\e}{\varepsilon}
\newcommand{\al}{\alpha}

\newcommand{\be}{\beta}
\newcommand{\G}{\Gamma}

\newcommand{\dd}{{\rm d}}

\usepackage{anysize}
\marginsize{2.5cm}{2.5cm}{2.5cm}{3cm}

\begin{document}
\begin{center}\emph{}
\LARGE
\textbf{About Convergence and Order of Convergence of some Fractional Derivatives}
\end{center}
\begin{center}
{ S. D. Roscani and L. D. Venturato}\\
CONICET - Depto. Matem\'atica,
FCE, Univ. Austral,\\
Paraguay 1950, S2000FZF Rosario, Argentina \\
(sroscani@austral.edu.ar, lventurato@austral.edu.ar)\\

\vspace{0.2cm}
\end{center}
\small

\noindent \textbf{Abstract: } In this paper we establish some convergence results for Riemann-Liouville, Caputo, and Caputo--Fabrizio fractional operators when the order of differentiation approaches one. We consider some errors given by $\left|\left| D^{1-\al}f -f'\right|\right|_p$ for p=1 and $p=\infty$ and we prove that for both Caputo and Caputo Fabrizio operators the order of convergence is a positive real $r \in (0,1)$. Finally, we compare the speed of convergence between Caputo and Caputo--Fabrizio operators obtaining that they a related by the Digamma function.



\section{Introduction}\label{Sec:Intro} 

In recent years, a discussion about different kinds of fractional operators has taken relevance. Questions like  which of them are worth calling ``Fractional Derivatives'' and which of them are not, or which properties must verify a fractional derivative have derived in recent publications (see e. g. \cite{AnTeMaKiAta:2016, BhaPa:2019, Giu:2018, Hi:2019, OrMa:2019,RoDu:2017}).
 As a consequence some classification criteria related to fractional derivatives had emerged. Works in this direction are, for example, \cite{HiLu:2019,TeMaOl:2019} where different properties for the fractional operators are established, which may vary depending on the spaces of function where they are applied. 
These classification criteria may vary from one author to another, but it is not disputed that a fractional derivative must  be a linear operator which converges to an ordinary derivative when the order of differentiation approaches to a positive integer number. That is,  if an operator $D^\alpha$ is  considered a fractional derivative, then it is linear and  $\lim\limits_{\alpha\rightarrow n}||D^\alpha f-f^{(n)}||=0$, for every $f$ belonging to an appropriated normed space of functions, $(X,||\cdot||)$.

Numerous operators were analyzed under the different proposed criteria. Two of the most outstanding fractional operators are the  Riemann-Liouville and the Caputo fractional derivatives. We will refer to them as RL derivatives and C derivatives. Both operators have been  extensively studied in \cite{Diethelm, Kilbas, Podlubny} and references therein.  In particular, applications to the theory of viscoelasticity or subdiffusion processes are described in  \cite{FM-Libro, MK:2000, Povstenko}. 
Another integrodifferential operator was recently defined in \cite{CaFa:2015}, called the Caputo-Fabrizio derivative (henceforth CF derivative). It is not difficult to check that this operator verifies the criteria proposed in  \cite{HiLu:2019, OtTr:2014}. In addition, the CF derivative is defined throughout a kernel without singularity, whereas RL y C derivatives are defined by integrodifferential operators with singular kernels.

The purpose of this paper is the study of some topics related to the convergence of C, CF and RL fractional derivatives to the ordinary derivative when $\al \nearrow 1$. In section 2 some basic definitions and results about convergence almost everywhere of the fractional operators mentioned before are presented. In section 3, we analyze the order of convergence when the fractional order of derivation goes to 1 in each case, by considering $L^1$ and $L^\infty$ norms. In particular, we obtain that the order for the C and CF derivatives are both lower than 1, whereas it is not possible to determine, in general, an order of convergence  for the RL  derivative in these norms. Finally, we compare the speed of convergence between C and CF for $L^1$ and $L^\infty$ norms, obtaining that it depends on the Digamma function and on the length of the interval where the norms are computed.

\section{Preliminaries}

Let us define the fractional derivatives involved in this work. Let $(a,b)\subset \bbR$  be a bounded interval (that is $-\infty<a<b<\infty$). 
\begin{defi}\label{defi frac} Let $\al \in (0,1)$. 
\begin{enumerate}[1)]
	\item\label{def-IntRL}   If $f \in L^1(a,b)$ we define the \textsl{fractional Riemann--Liouville integral of order  $\alpha$} as
$$_{a}I^{\alpha}f(t)=\frac{1}{\Gamma(\alpha)}\int^{t}_{a}f(\tau)(t-\tau)^{\alpha-1} \dd\tau. $$
\item\label{def-DerRL} If $ f\in W^{1,1}(a,b)=\left\{f \in L^1(a,b) \, /\, \exists g \in L^1(a,b) \text{ such that } \int_a^bf\varphi'= - \int_a^bg\varphi, \,\forall \varphi \in C^1_c(a,b)\right\}$, we define the 
\textsl{fractional Riemann--Liouville  derivative of order $\alpha$ }  as   
$$ ^{RL}_{a}D^{\alpha}f(t)= \left[ \frac{\dd}{\dd t} \, _{a}I^{1-\alpha}f  \right] (t) =\frac{1}{\Gamma(1 - \alpha)}\frac{d}{dt}\int^{t}_{a}f(\tau)(t-\tau)^{-\alpha} \dd\tau. $$
\item\label{def-DerC} If $f\in W^{1,1}(a,b)$ we define the \textsl{fractional Caputo derivative of order  $\al$}  as
$$\,^C_{a} D^{\alpha}f(t)= \left[ \, _{a}I^{1-\alpha}\left(\frac{\dd}{\dd t}f\right)  \right] (t) = \frac{1}{\Gamma(1-\al)}\displaystyle\int^{t}_{a} f'(\tau) (t-\tau)^{-\al} \dd\tau$$
\end{enumerate}
\end{defi}
\begin{prop}\label{relacion RL-C}\cite{Kilbas} If $0<\al<1$ and  $f\in W^{1,1}(a,b)$ then    
$$ ^{RL}_{a}D^{\al}f (t)=\frac{f(a)}{\G(1-\al)}(t-a)^{-\al}+\, ^C_{a} D^{\alpha}f(t).$$
\end{prop}

\begin{defi}\label{def-DerCF}
Let $f$ be a function in $W^{1,1}(a,b)$. We define the 
\textsl{fractional Caputo-Fabrizio  derivative of order $\alpha$ }  as
\begin{equation}
{^C}{^F_a}D^\alpha f(t)=\frac{1}{1-\alpha}\int_a^t f'(t) e^{-\frac{\alpha}{1-\alpha}(t-\tau)}d\tau.
\end{equation}
\end{defi}

\begin{prop}\label{prop-deriv-pot} If $f(t)=(t-a)^\gamma$ is a real function defined in $[a,b]$ ($\gamma >0$) and $\al \in (0,1)$ then 
\begin{enumerate}[a)]
	\item ${^C_a} D^{\alpha}f(t)=\frac{\G(\gamma + 1)}{\G(\gamma -\al+1)}(t-a)^{\gamma-\al}$. 
	\item  ${^C}{^F_a}D^\alpha f(t)=\frac{\gamma}{\al}(t-a)^{\gamma-1}\left[1-\Gamma(\gamma) \mathcal{E}_{1,\gamma}\left(-\frac{\al}{1-\al}(t-a)\right) \right]$, where $\mathcal{E}_{\rho,\omega}(\cdot)$ is the Mittag–Leffler function defined for every $t\in \bbR$ by
$\mathcal{E}_{\rho,\omega}(t)=\sum\limits_{k=0}^{\infty}\frac{t^k}{\G(\rho k+\omega)}.$
\end{enumerate}

\end{prop}
\proof See \cite{Podlubny} for the proof of $a)$ and \cite[Prop 4]{RoTaVe:2019} for $b)$.  
\endproof

Henceforth, given a function  $f$ defined in an interval $(a,b)$ its extension by zero to $\mathbb{R}$  will be considered if the context requires a hole definition.

\vspace{0.3cm}

In the aim to give in the next section some general results, we define now a general linear operator which coincides, under certain assumptions, with the  fractional derivatives defined above.  

\begin{defi}\label{Def-Gen-DF}
Let $(a,b)\subset \mathbb{R}$ and let $h:\bbR^+ \times(0,1)\rightarrow \mathbb{R}$ be a function such that \linebreak $h(\cdot,\beta)\in  W^{1,1}(\bbR^+)$ and $h(\cdot,\beta)\in  L^1(\bbR^+)$ uniformly in $(0, \be_0],  \beta_0\in (0,1)$. We define the operator \linebreak ${_a^h}D^{1-\beta}:W^{1,1}(a,b)\rightarrow W^{1,1}(a,b)$ by
\begin{equation}\label{Gen-DF}\displaystyle {_a^h}D^{1-\beta} g(t):=(g'* h(\cdot,\beta))(t) \qquad a.e. \, \text{in } \, (a,b).\end{equation}
\end{defi}

\begin{remark}
By considering the kernels $h_C(t,\alpha)=\frac{1}{\Gamma(\alpha)}t^{-(1-\alpha)}$ and $h_{CF}(t,\alpha)=\frac{e^{-\frac{1-\alpha}{\alpha}t}}{\alpha}$ defined for $t>0$, we recover the C and CF derivatives   of order $1-\alpha$ given in Definition \ref{defi frac}-\ref{def-DerC} and Definition \ref{def-DerCF} respectively.   
\end{remark}

The following are classical results about almost everywhere convergence of  C and RL derivatives  to the classical differential operator $\frac{\dd}{\dd t}$, when the order of convergence tends to 1.

\begin{prop}\label{conv-RL1} The following limits hold.
\begin{enumerate}[a)]
\item If $f\in W^{1,1}(a,b)$, then $$\lim\limits_{\al\nearrow 1}\,^{R}{^L_a} D^{\alpha}f(t)= \frac{\dd}{\dd t}f(t), \quad a.e. \, \text{ in } \, (a,b).$$
\item If $f\in C^1[a,b]$, then
$$\displaystyle\lim_{\alpha\nearrow 1} \,^{R}{^L_a}D^{\alpha} f(t)= \frac{\dd}{\dd t} f(t) \qquad  \text{ for every } \, t \in  (a,b].$$
\end{enumerate}
\end{prop}
The proof of $a)$ follows from \cite[Theorem 2.6]{Samko} and Proposition \ref{relacion RL-C}. For $b)$ see \cite[Theorem 2.20]{Diethelm}. As a consequence of Propositions \ref{relacion RL-C} and \ref{conv-RL1} the next result holds.

\begin{prop}\label{conv-C1} The following limits hold. 
\begin{enumerate}[a)]
\item If $f\in W^{1,1}(a,b)$, then $$\lim\limits_{\al\nearrow 1}\,^C_{a} D^{\alpha}f(t)= \frac{\dd}{\dd t}f(t), \quad a.e. \, \text{ in } \, (a,b).$$
\item If $f\in C^1[a,b]$, then
$$\displaystyle\lim_{\alpha\nearrow 1}{^C_a}D^{\alpha} f(t)= \frac{\dd}{\dd t} f(t) \qquad  \text{ for every } \, t \in  (a,b].$$
\end{enumerate}
\end{prop} 

For the CF derivative we present the next result, which is a generalization of the one obtained in \cite{RoTaVe:2019}. 
\begin{prop}The following limits hold. 

\begin{enumerate}[a)]
\item If $f\in W^{1,1}(a,b)$, then$$\displaystyle\lim_{\alpha\nearrow 1}{^C}{^F_a}D^{\alpha} f(t)= \frac{\dd}{\dd t} f(t) \qquad  a.e. \text{ in } (a,b).$$
\item If $f\in C^2[a,b]$, then
$$\displaystyle\lim_{\alpha\nearrow 1}{^C}{^F_a}D^{\alpha} f(t)= \frac{\dd}{\dd t} f(t) \qquad  \text{ for every } \, t \in  (a,b].$$
\end{enumerate}
\end{prop}

\begin{proof} a)
Let $f \in W^{1,1}(a,b)$ be. Then $f' \in L^1(a,b)$ and by classical density results in $L^1(a,b)$ (see e.g. \cite{Bartle})   
there exists a simple function, which will be called $g_\e'$ by an abuse of language, such that $g_\varepsilon' (t)=\sum\limits_{i=1}^n q_i \chi_{[a_i,b_i]}(t)$,   where $b_i\leq a_{i+1}$ for every $i=1,...,n-1$ and
\begin{equation}\label{eq0}
||f'-g'_{\varepsilon}||_{L^1 (a,b)}<\frac{\varepsilon}{2}.
\end{equation}
Now for every $t\in[a,b]$, let  $g_\e(t)=\int_a^t g_\e' $ be.
Then if  $t\in[a_k,b_k]$, for any $k$ given
it follows that 
\begin{equation}\label{eq1}
 {^C}{^F_a}D^{\alpha} g_\e(t)=\sum\limits_{i=1}^{k-1}q_i \frac{1}{\alpha} \left(e^{-\frac{\alpha}{1-\alpha}(t-b_i)}-e^{-\frac{\alpha}{1-\alpha}(t-a_i)}\right)+q_k \frac{1}{\alpha} \left(1-e^{-\frac{\alpha}{1-\alpha}(t-a_k)}\right).
\end{equation}
Taking the limit when $\al \nearrow 1$ in  \eqref{eq1}, we have that 
\begin{equation*}
\lim\limits_{\alpha\nearrow 1} {^C}{^F_a}D^{\alpha} g_\e(t)=q_k=g_\varepsilon' (t).
\end{equation*}
For every $t\in (a,b)$ and $\alpha\geq \frac{1}{2}$ the following estimations hold
$$|g_\e'(t)|\leq \sum\limits_{i=i}^n q_i \quad \text{ and } \quad \left| {^C}{^F_a}D^{\alpha} g_\e(t)\right|\leq \sum\limits_{i=i}^k 2q_i,$$
thus the Lebesgue Convergence Theorem can be applied to compute the next limit 
\begin{equation}\label{eq2}
\lim\limits_{\alpha\nearrow 1} \left|\left|{^C}{^F_a}D^{\alpha} g_\e-g_\varepsilon'\right|\right|_{L^1 (a,b)}=\lim\limits_{\alpha\nearrow 1} \int_a^b \left|{^C}{^F_a}D^{\alpha} g_\e(t)-g_\varepsilon'(t)\right|\dd t =0.
\end{equation}
By the other side, 
\begin{equation}\label{eq00.4}
\left|\left|{^C}{^F_a}D^{\alpha} f- {^C}{^F_a}D^{\alpha} g_\e\right|\right|_{L^1 (a,b)}\leq \int_a^b \left|f'(\tau) -g_\varepsilon' (\tau) \right| \frac{1-e^{-\frac{\alpha}{1-\alpha}(b-\tau)}}{\alpha} \dd \tau\leq \frac{1}{\alpha}||f'- g_\varepsilon'||_{L^1 (a,b)}.
\end{equation}
where  Fubini's Theorem has been applied due to the fact that $f'(\tau) -g_{\varepsilon}' (\tau) \in L^1(a,b)$ and the boundedness of the Caputo-Fabrizio kernel. \\
From  \eqref{eq0} and \eqref{eq00.4} we have that
\begin{equation}\label{eq5}
\left|\left|{^C}{^F_a}D^{\alpha} f- {^C}{^F_a}D^{\alpha} g_\e\right|\right|_{L^1 (a,b)}< \frac{\varepsilon}{2\alpha}.
\end{equation}
Finally, from \eqref{eq0} and \eqref{eq5} it holds that
\begin{equation}\label{eq00.6}
\left|\left|{^C}{^F_a}D^{\alpha} f- f'\right|\right|_{L^1 (a,b)}\leq \frac{\varepsilon}{2\al}+\left|\left|{^C}{^F_a}D^{\alpha} g_\e - g_\varepsilon'\right|\right|_{L^1 (a,b)}+\frac{\varepsilon}{2}.
\end{equation}
Taking the limit when $\al\nearrow 1$, in  \eqref{eq00.6} we conclude that 
\begin{equation*}
\begin{split}
\lim\limits_{\alpha\nearrow 1}\left|\left|{^C}{^F_a}D^{\alpha} f- f'\right|\right|_{L^1 (a,b)}\leq \varepsilon, \qquad \text{for every } \varepsilon>0,
\end{split}
\end{equation*}
and in consequence item $a)$ is proved. The proof of $b)$ is given in \cite{RoTaVe:2019}.\end{proof}
As we said before, the goal of this paper is to analyze the convergence and the ``speed'' of convergence of the mentioned fractional derivatives in different norms. Are they strongly different taking into account that C and RL derivatives  are defined in terms of singular kernels while CF derivative is defined for a no singular kernel?  We will see in the next section that the answer is no. 

In Figures 1 and 2, functions $f(t)=t+1$ and $g(t)=\cos t$ are compared, and it can be seen how the fractional RL, CF  and C derivatives converge pointwise in $(0,1)$ to $f'$ and $g'$ respectively.

\begin{figure}[h!]
\begin{minipage}[b]{0.5\linewidth}
\centering
\includegraphics[width=8cm]{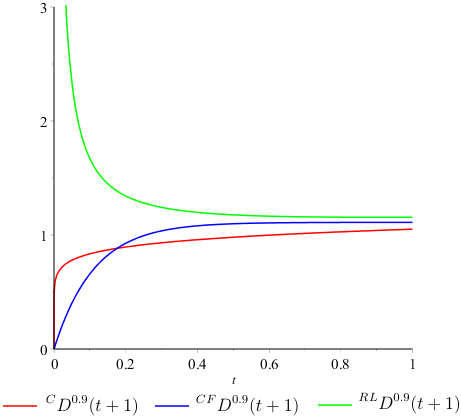}
\end{minipage}
\hspace{0.5cm}
\begin{minipage}[b]{0.5\linewidth}
\centering
\includegraphics[width=8cm]{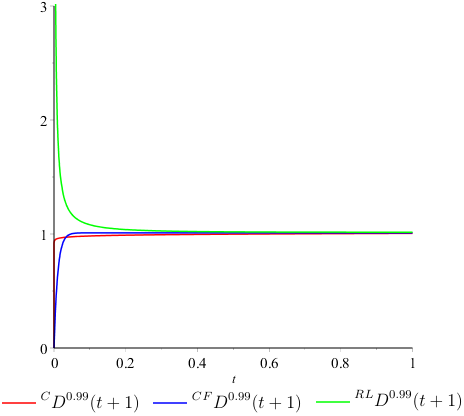}
\end{minipage}
\caption{Some fractional derivatives of $f(t)=t+1$}
\label{figura001}
\end{figure}
\begin{figure}[h!]
\begin{minipage}[b]{0.5\linewidth}
\centering
\includegraphics[width=8cm]{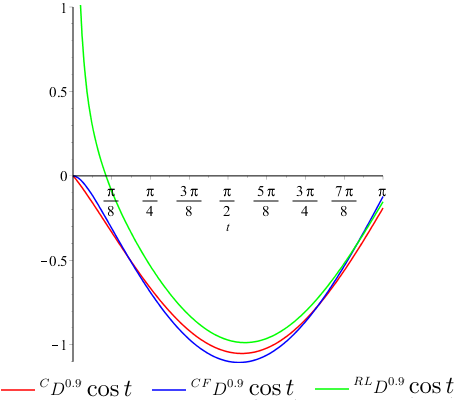}
\end{minipage}
\hspace{0.5cm}
\begin{minipage}[b]{0.5\linewidth}
\centering
\includegraphics[width=8cm]{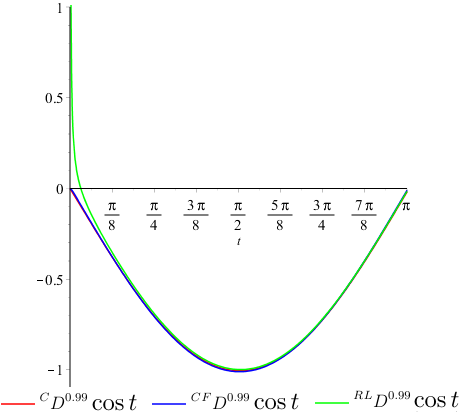}
\end{minipage}
\caption{Some fractional derivatives of $f(t)=\cos(t)$}
\label{figura002}
\end{figure}


\section{Caputo and Caputo-Fabrizio Convergence}
\subsection{Order of convergence}

In  this section we study the order of convergence respect on the parameter related to the order of differentiation,  for different $L^p$ norms. In this direction we define the following general error expression associated to a general fractional operator.

\begin{defi}
Let $(a,b)\subset \mathbb{R}$, and  $f\in W^{1,1}(a,b)$,  $1\leq p\leq \infty$. We define the error estimate in the $L^p$ norm associated to the fractional derivative as
  \begin{equation}\label{Error-Lp}
	\begin{array}{rl}
	E_{f,p}\, :\, (0,1) & \rightarrow \mathbb{R}_0^+\\
\beta & \rightarrow 	E_{f,p}(\beta)=||{_a}D^{1-\beta}f -f'||_{L^p(a,b)}.
\end{array}
\end{equation}
\end{defi}

From now on we will denote by ${_a}D^{1-\beta}$ to refer to the  fractional derivative of RL, C or CF type (even to the general operator defined in (\ref{Gen-DF}), where the superscript may be omitted if it does not leads to doubts).   

\begin{note}
By observing the graphics presented in Figures \ref{figura001} and \ref{figura002}, we can suppose that the error estimate for the RL operator will not be, in general, an estimate that converges to zero.   In fact, if we consider  $f(t)\equiv 1$ and $a=0$, Propositions \ref{relacion RL-C} and \ref{prop-deriv-pot} yields that
\begin{equation}\label{DRL-cte}
{^R}{^L}D^{1-\beta} f(t)=\frac{t^{\beta-1}}{\Gamma(\beta)}.
\end{equation}
By replacing $(\ref{DRL-cte})$ in $(\ref{Error-Lp})$ we have
$$E_{f,1}(\beta)=\left|\left|{^R}{^L}D^{1-\beta} f- f'\right|\right|_{L^1(0,b)}=\left|\left|{^R}{^L}D^{1-\beta} f\right|\right|_{L^1(0,b)}=\frac{b^\beta}{\Gamma(\beta+1)}.$$
Then 
 $$\lim\limits_{\beta \rightarrow 0^+}\frac{E_{f,1}(\beta)}{\beta^r}=+\infty.$$
Analogously,  
$$\lim\limits_{\beta \rightarrow 0^+}\frac{E_{f,\infty}(\beta)}{\beta^r}=+\infty.$$
We conclude that the RL errors $E_{f,1}(\beta)$ and $E_{f,\infty}(\beta)$ are not, in general,  $o(\beta^r)$ or  $O(\beta^r)$ for every $r>0$.
\end{note}

The next Lemma states that, if we can estimate the rates $\frac{E_{g,1}(\beta)}{\beta^r}$ for every  $g$ such that $g' \in D$ (where $D$ is a dense set contained in $L^1(a,b)$), then we can estimate the rate  $\frac{E_{f,1}(\beta)}{\beta^r}$ for any function $f\in W^{1,1}(a,b)$.

\begin{lemma}\label{acotgeneral}
Let $(a,b)\subset \bbR$ and ${_a}D^{1-\beta}$ a fractional operator defined in $(\ref{Gen-DF})$. Suppose that there exist a dense subset $D\subset L^1(a,b)$ and a fix number $r>0$  such that 
\begin{equation}\label{hipot1}
E_{g,1}(\beta)=O(\beta^r)\quad  (\text{resp.} \,E_{g,1}(\beta)=o(\beta^r)) , \quad  \, \beta\rightarrow 0^+ \quad \forall g\in W^{1,1}(a,b), g' \in D.
\end{equation}
 Then, 
\begin{equation}\label{conclu1}
E_{f,1}(\beta)=O(\beta^r)\quad  (\text{resp.} \, E_{f,1}(\beta)=o(\beta^r)), \quad \beta\rightarrow 0^+ \quad  \forall f\in W^{1,1}(a,b).
\end{equation}
\end{lemma}

\begin{proof}
Let $f\in W^{1,1}(a,b)$ and $\varepsilon>0$ be. Then $f' \in L^1(a,b)$  and there exist a function  $g'_\e$ in D, (abuse of language again), such that
\begin{equation}\label{densidad}
\left|\left|g'_\e-f'\right|\right|_{L^1(a,b)}<\varepsilon.
\end{equation}
Thus, if we set $g_\e(t) =\int_a^t g'_\e$  for every $t\in (a,b)$ it holds that  
\begin{equation}\label{desnorma}
E_{f,1}(\beta)\leq \left|\left|{_a}D_t^{1-\beta} (f- g_\varepsilon)\right|\right|_{L^1(a,b)}+E_{g_\varepsilon,1}(\beta)+\left|\left|g'_\varepsilon-f'\right|\right|_{L^1(a,b)}.
\end{equation}
By applying definition (\ref{Gen-DF}) and Young's inequality   we get
\begin{equation}\label{desigyoung}
\begin{split}
\left|\left|{_a}D_t^{1-\beta} (f- g_\varepsilon)\right|\right|_{L^1(a,b)}&=\left|\left|(f- g_\varepsilon)'*h(\cdot, \beta)\right|\right|_{L^1(a,b)}\leq  \left|\left|f'- g'_\varepsilon \right|\right|_{L^1(a,b)}\left|\left|h(\cdot,\beta)\right|\right|_{L^1(a,b)}.
\end{split}
\end{equation}
And using the uniformly boundedness of $h$ and (\ref{densidad}) gives that 
\begin{equation}\label{des00}
\left|\left|{_a}D_t^{1-\beta} (f- g_\varepsilon)\right|\right|_{L^1(a,b)}\leq  K\e.
\end{equation}
By applying inequalities  \eqref{hipot1} \eqref{densidad}  and   \eqref{des00}, to \eqref{desnorma}, it  yields that 
\begin{equation}\label{desnorma2}
\begin{split}
E_{f,1}(\beta)&\leq (K+1)\varepsilon +C\beta^r  \qquad \be\rightarrow 0.
\end{split}
\end{equation}
From the  discretionary choice of $\e$, the thesis holds.
\end{proof}

\begin{remark}\label{coroconvgeneral}
An analogous result to the given in  Lemma \ref{acotgeneral} is obtained by replacing the  $||\cdot||_{L^1(a,b)}$ norm by the  $||\cdot||_p$ norm, for $1\leq p\leq \infty$, holds. The proof is identical due to the validity of Young's inequality.
\end{remark}
Let us compute the order of convergence for the CF derivative.

\begin{theo}\label{OrderCapFab}
Let  $f\in W^{1,1}(a,b)$ and  ${_a}D_t^{1-\beta}={^C}{^F_a}D^{1-\beta}$. Then,
$$E_{f,1}(\beta)=o(\beta^r),\quad \beta\rightarrow 0^+ ,\,\, \forall r\in (0,1),$$
and 
$$E_{f,1}(\beta)=O(\beta),\quad \beta\rightarrow 0^+.$$
\end{theo}

\begin{proof}

Let $h:\bbR^+\times(0,1)\rightarrow \mathbb{R}$ be defined by $h(t,\beta)=\frac{e^{-\frac{1-\beta}{\beta}t}}{\beta}$. Note that  ${_a^h}D_t^{1-\beta}={^C}{^F_a}D^{1-\beta}$ because $h$ is an admissible kernel in Definition \ref{Def-Gen-DF}.

Now, let $\varepsilon>0$ and  $f\in W^{1,1}(a,b)$ be. Let $D$ be the  subset of simple functions  dense in  $L^1(a,b)$ described in Proposition 4 and, by following the previous notation, let the function
 $g'_\varepsilon(t)=\sum\limits_{i=1}^n q_i \chi_{[a_i,b_i]}(t) \in D$ be such that 
\begin{equation}
||f'-g'_\varepsilon||_{L^1(a,b)}<\varepsilon.
\end{equation}
Note that
\begin{equation}\label{acotpoli}
||g_\varepsilon'||_{L^1(a,b)}\leq ||g_\varepsilon'-f'||_{L^1(a,b)}+||f'||_{L^1(a,b)}<\varepsilon+||f'||_{L^1(a,b)}.
\end{equation}
Then if we set $g_\varepsilon(t) =\int_a^t g_\varepsilon'$ and  integrate by parts, it holds that the error estimate in the interval $(a_k,b_k)$ verifies that 
\begin{equation}\label{desigorden}
\begin{split}
E_{g_\varepsilon\chi_{[a_k,b_k]},1}(\beta)&=\int_{a_k}^{b_k}\left|\sum\limits_{i=1}^{k-1} \frac{q_i}{1-\beta} \left(e^{-\frac{1-\beta}{\beta}(t-b_i)}-e^{-\frac{1-\beta}{\beta}(t-a_i)}\right)+ \frac{q_k}{1-\beta} \left(1-e^{-\frac{1-\beta}{\beta}(t-a_k)}\right)-q_k\right|\dd t\\
&\leq \beta^r\left[\sum\limits_{i=1}^{k-1}\frac{\left|q_i\right| \beta^{1-r}}{(1-\beta)^2} \left(e^{-\frac{1-\beta}{\beta}(a_k-b_i)}-e^{-\frac{1-\beta}{\beta}(b_k-b_i)}+ e^{-\frac{1-\beta}{\beta}(a_k-a_i)}-e^{-\frac{1-\beta}{\beta}(b_k-a_i)}\right)\right.\\
&\quad\left.+\left|q_k\right|(b_k-a_k) \frac{\beta^{1-r}}{1-\beta}+\left| q_k\right|\frac{\beta^{1-r}}{(1-\beta)^2}\left(1-e^{-\frac{1-\beta}{\beta}(b_k-a_k)}\right)\right].
\end{split}
\end{equation}
Clearly, the last inequality in \eqref{desigorden} tends to 0  when $\beta\searrow 0$ if $r\in (0,1)$. 
 Being  \linebreak $E_{g_\varepsilon,1}(\beta)=\sum\limits_{k=1}^n E_{g\chi_{[a_k,b_k]},1}(\beta),$ we conclude that 
$$E_{g_\varepsilon,1}(\beta)=o(\beta^r),\quad \forall r\in (0,1).$$
Thereupon, Lemma \ref{acotgeneral} gives that
$$E_{f,1}(\beta)=o(\beta^r),\quad \forall r\in (0,1).$$
Consider now $r=1$. Then  \eqref{desigorden} becomes
\begin{equation}\label{desigbeta}
\begin{split}
E_{g_\varepsilon\chi_{[a_k,b_k]},1}(\beta)&\leq \beta\left[\sum\limits_{i=1}^{k-1}\frac{\left|q_i\right|}{(1-\beta)^2} \left(e^{-\frac{1-\beta}{\beta}(a_k-b_i)}-e^{-\frac{1-\beta}{\beta}(b_k-b_i)}+ e^{-\frac{1-\beta}{\beta}(a_k-a_i)}-e^{-\frac{1-\beta}{\beta}(b_k-a_i)}\right)\right.\\
&\quad\left.+(b_k-a_k) \frac{\left|q_k\right|}{1-\beta}+\frac{\left| q_k\right|}{(1-\beta)^2}\left(1-e^{-\frac{1-\beta}{\beta}(b_k-a_k)}\right)\right].
\end{split}
\end{equation}
Define
$$A_k=\sum\limits_{i=1}^{k-1}\frac{\left|q_i\right|}{(1-\beta)^2} \left(e^{-\frac{1-\beta}{\beta}(a_k-b_i)}-e^{-\frac{1-\beta}{\beta}(b_k-b_i)}+ e^{-\frac{1-\beta}{\beta}(a_k-a_i)}-e^{-\frac{1-\beta}{\beta}(b_k-a_i)}\right), \quad k=1,\dots,n.$$
Hence, applying the mean value theorem to the last term in brackets in \eqref{desigbeta}, we get
\begin{equation}\label{betaA}
\begin{split}
E_{g_\varepsilon\chi_{[a_k,b_k]},1}(\beta)
&=\beta\left[A_k+(b_k-a_k) \frac{\left|q_k\right|}{1-\beta}\left(1+\frac{e^{-\frac{1-\beta}{\beta}(b_k-c_k)}}{\beta}\right)\right],
\end{split}
\end{equation}
where $a_k<c_k<b_k$. Being $\lim\limits_{\beta\rightarrow 0^+} A_k =0$ and $\lim\limits_{\beta\rightarrow 0^+}\frac{e^{-\frac{1-\beta}{\beta}(b_k-c_k)}}{\beta}=0$, $\forall k=1,2,\cdots,n$, we conclude that
\begin{equation}\label{betaB}
E_{g_\varepsilon,1}(\beta)=\sum\limits_{k=1}^{n}E_{g_\varepsilon\chi_{[a_k,b_k]},1}(\beta)\leq\beta\left[K_1+K_2\sum\limits_{k=1}^{n}\left|q_k\right|(b_k-a_k)\right]\leq\beta K,
\end{equation}
where the last inequality is valid because
\begin{equation*}
\sum\limits_{k=1}^{n}\left|q_k\right|(b_k-a_k)=||g_\varepsilon||_{L^1(a,b)}\leq \varepsilon +||f'||_{L^1(a,b)}\leq 1+||f'||_{L^1(a,b)},\quad \forall \varepsilon \in (0,1).
\end{equation*}
Thus, by Lemma \ref{acotgeneral}, $E_{f,1}(\beta)=O(\beta)$ for every $f\in W^{1,1}(a,b)$.
\end{proof}

From the equivalence of the $L^\infty$ and $L^1$ norms in bounded intervals the next theorem follows.

\begin{theo}
Let  $f\in W^{1,1}(a,b)$ and  ${_a}D_t^{1-\beta}={^C}{^F_a}D^{1-\beta}$. Then,
$$E_{f,\infty}(\beta)=o(\beta^r),\quad \beta\rightarrow 0^+ ,\,\, \forall r\in (0,1),$$
and 
$$E_{f,\infty}(\beta)=O(\beta),\quad \beta\rightarrow 0^+.$$
\end{theo}

\medskip
\begin{remark}\label{remejem}
Let us see that $E_{f,1}(\beta)$ and $E_{f,\infty}(\beta)$ are not in general $o(\beta)$ when $\be\rightarrow 0^+$. In fact, let $f_1(t)=t$ and  $f_2(t)=e^t$ be functions defined in $t\in [0,b]$ .

Being $\left|{^C}{^F}D^{1-\beta} f_1(\cdot) - f_1'(\cdot)\right|$  a continuous function in $[0,b]$, we have
\begin{equation}\label{error-CF-t-infty}
\begin{split}
E_{f_1,\infty}(\beta)=
\left|\left|{^C}{^F}D^{1-\beta} f_1 - f_1'\right|\right|_{L^\infty(0,b)}\geq \left|{^C}{^F}D^{1-\beta} f_1(0) - f_1'(0)\right|=1.
\end{split}
\end{equation}
We conclude that
$$\lim\limits_{\beta \rightarrow 0^+}\frac{E_{f_1,\infty}(\beta)}{\beta}\neq 0.$$
By the other side, it is easy to see that  
\begin{equation*}
E_{f_2,1}(\beta)=\left|\left|{^C}{^F}D^{1-\beta} f_2 - f_2'\right|\right|_{L^1(0,b)}=\frac{\beta}{1-\beta}\left(1-e^{-\frac{1-\beta}{\beta}b}\right),
\end{equation*}
then,
$$\lim\limits_{\beta \rightarrow 0^+}\frac{E_{g,1}(\beta)}{\beta}=1\neq 0.$$
\end{remark}

We turnout now the analysis to the C derivative.  

\begin{theo}\label{OrderCap}
Let  $f\in W^{1,1}(a,b)$ and  ${_a}D_t^{1-\beta}={^C_a}D^{1-\beta}$. Then,
$$E_{f,1}(\beta)=o(\beta^r),\quad \beta\rightarrow 0^+ ,\,\, \forall r\in (0,1),$$
and 
$$E_{f,1}(\beta)=O(\beta),\quad \beta\rightarrow 0^+.$$
\end{theo}

\begin{proof}

Let $h:\bbR^+\times(0,1)\rightarrow \mathbb{R}$ defined by $h(t,\beta)=\frac{t^{-(1-\beta)}}{\Gamma(\beta)}$. Note that $h$ is an admissible kernel in Definition \ref{Def-Gen-DF}, moreover,  ${_a^h}D_t^{1-\beta}={^C_a}D^{1-\beta}$. 

Let $\varepsilon>0$ and  $f\in W^{1,1}(a,b)$ be. Let $D$ be a dense subset in  $L^1(a,b)$ described in Proposition 4. Reasoning like in Theorem \ref{OrderCapFab}, consider the function
 $g'_\varepsilon(t)=\sum\limits_{i=1}^n q_i \chi_{[a_i,b_i]}(t) \in D$ such that 
\begin{equation}
||f'-g'_\varepsilon||_{L^1(a,b)}<\varepsilon.
\end{equation}
Then, \eqref{acotpoli} holds.
Let us estimate the error in the interval $(a_k,b_k)$. Setting $g_\varepsilon(t) =\int_a^t g_\varepsilon'$ and applying Proposition \ref{prop-deriv-pot}  it holds that

\begin{equation}\label{escalonC}
E_{g_\varepsilon\chi_{[a_k,b_k]},1}(\beta)=\beta^r\frac{|q_k|}{\Gamma(1+\beta)}\int_{a_k}^{b_k} \left| \frac{(t-a_k)^\beta-\Gamma(1+\beta)}{\beta^r}\right|\dd t.
\end{equation}
And,
$$\lim\limits_{\beta\rightarrow 0^+}\int_{a_k}^{b_k} \left| \frac{(t-a_k)^\beta-\Gamma(1+\beta)}{\beta^r}\right|\dd t= 0, \text{ if } r \in (0,1).$$
Then, $\lim\limits_{\beta\rightarrow 0^+}=E_{g_\varepsilon\chi_{[a_k,b_k]},1}(\beta)=0$. Being   $E_{g_\varepsilon,1}(\beta)=\sum\limits_{k=1}^n E_{g\chi_{[a_k,b_k]},1}(\beta),$ we conclude that 
$E_{g_\varepsilon,1}(\beta)=o(\beta^r),\quad \forall r\in (0,1)$.\\
In consequence, Lemma \ref{acotgeneral} gives that $E_{f,1}(\beta)=o(\beta^r),\quad \forall r\in (0,1)$.\\
Consider the case $r=1$. From \eqref{escalonC} we have
\begin{equation}\label{limit1}
\begin{split}
E_{g_\varepsilon\chi_{[a_k,b_k]},1}(\beta)&=\beta\frac{|q_k|}{\Gamma(1+\beta)}\int_{a_k}^{b_k} \left| \frac{(t-a_k)^\beta-\Gamma(1+\beta)}{\beta}\right|\dd t.
\end{split}
\end{equation}
Now,
\begin{equation*}
\begin{split}
\lim\limits_{\beta\rightarrow 0^+}\frac{(t-a_k)^\beta -\Gamma(\beta+1)}{\beta}= \ln(t-a_k)-\Gamma'(1).
\end{split}
\end{equation*}
Thus, by Lebesgue convergence Theorem, we have that 
$$\lim\limits_{\beta\rightarrow 0^+}\int_{a_k}^{b_k} \left| \frac{(t-a_k)^\beta-\Gamma(1+\beta)}{\beta}\right|\dd t=\int_{a_k}^{b_k} |\ln(t-a_k)-\Gamma'(1)|<\infty,$$
hence $E_{g_\varepsilon,1}(\beta)=O(\beta)$ for every $g$ such that $g' \in D$. By applying Lemma \ref{acotgeneral}the thesis holds.\end{proof}

\begin{theo}
Let $f\in \{g\in W^{1,p}(a,b): g'\in L^\infty (a,b)\}$ and  ${_a}D_t^{1-\beta}={^C_a}D^{1-\beta}$. Then,
$$E_{f,\infty}(\beta)=o(\beta^r),\quad \beta\rightarrow 0^+ ,\,\,  \forall r\in(0,1),$$
$$E_{f,\infty}(\beta)=O(\beta),\quad \beta\rightarrow 0^+.$$
\end{theo}

\begin{remark}
It is easy to see that $E_{f,1}(\beta)$ and $E_{f,\infty}(\beta)$ are not in general $o(\beta)$ when $\be\rightarrow 0^+$. In fact, let $g(t)=|t-1|$ be defined in $t\in[0,2]$. Thus,
\begin{equation*}
{^C}D^{1-\beta} g(t)=\begin{cases}-\frac{t^\beta}{\Gamma(\beta+1)} & \text{ si } t\in [0,1],\\ \frac{2(t-1)^\beta-t^\beta}{\Gamma(\beta+1)} & \text{ si } t\in (1,2]\end{cases}
\end{equation*}
and
\begin{equation*}
\begin{split}
E_{g,1}(\beta)&=\left|\left|{^C}D^{1-\beta} g - g'\right|\right|_{L^1(0,2)}= \int\limits_0^1\left|-\frac{t^\beta}{\Gamma(\beta+1)}+1\right|\dd t+\int\limits_1^2\left|\frac{2(t-1)^\beta-t^\beta}{\Gamma(\beta+1)}-1\right|\dd t\\
& \geq \int\limits_1^2 1-\frac{2(t-1)^\beta-t^\beta}{\Gamma(\beta+1)}-1\dd t=\frac{\Gamma(\beta+2)-3+2^{\beta+1}}{\Gamma(\beta+2)}\geq 2(\Gamma(\beta+2)-3+2^{\beta+1}).
\end{split}
\end{equation*}
Being $\beta>0$,
\begin{equation*}
\begin{split}
\lim\limits_{\beta \rightarrow 0^+}\frac{E_{g,1}(\beta)}{\beta}&\geq \lim\limits_{\beta \rightarrow 0^+}2\frac{\Gamma(\beta+2)-3+2^{\beta+1}}{\beta}
=2(\Gamma'(2)+\ln 2)\neq 0,
\end{split}
\end{equation*}

Observing that
$$E_{g,1}(\beta)=\left|\left|{^C}D^{1-\beta} g - g'\right|\right|_{L^1(0,2)}\leq 2\left|\left|{^C}D^{1-\beta} g - g'\right|\right|_{L^\infty(0,2)}=2 E_{g,\infty}(\beta),$$
it's clear that
$$\lim\limits_{\beta \rightarrow 0^+}\frac{E_{g,\infty}(\beta)}{\beta}\geq \Gamma'(2)+1\neq 0.$$

\end{remark}

\begin{remark} It is worth noting that Theorems \ref{OrderCapFab} and \ref{OrderCap}  give us a similar order of convergence for CF and C derivatives respectively. It is interesting that the difference between the kernels (the first one non-singular and the second one singular!) had not substantially accelerate (or decelerate)  the speed of convergence.
\end{remark}

\subsection{Some comments about the speed of convergence}

From the results obtained in the previous section it is natural to wonder: Does the C derivative converges to the ordinary derivative faster than the CF derivative? Or conversely, does the CF derivative converges to the ordinary derivative faster than the C derivative?\\

The next two examples are presented to compare the convergence in the $L^\infty$ norm. 

\begin{example}\label{ex1}

Consider the function $f_1(t)=t$  defined in the interval $[0,1]$. 
From  \eqref{error-CF-t-infty}, we have that 

$$E_{f_1,\infty}(\beta)=\left|\left|{^C}{^F}D^{1-\beta} f_1 - f_1'\right|\right|_{L^\infty(0,1)}=k\geq 1, $$
since $\left|{^C}{^F}D^{1-\beta} f_1(0) - f_1'(0)\right|\geq 1$.
By the other side,
$$\left|\left|{^C}D^{1-\beta} f_1 - f_1'\right|\right|_{L^\infty(0,1)}=\max\limits_{t\in [0,1]}\left|1-\frac{t^\be}{\G(1+\be)} \right|= 
1.$$
Hence
$$\lim\limits_{\beta\rightarrow 0^+}\frac{\left|\left|{^C}{^F}D^{1-\beta} f_1 - f_1'\right|\right|_{L^\infty(0,1)}}{\left|\left|{^C}D^{1-\beta} f_1 - f_1'\right|\right|_{L^\infty(0,1)}}=k.$$
Then, in this case, the C derivative converges to the ordinary derivative faster than the CF derivative.\end{example}
\begin{example}\label{ex2}
Consider now the function $g(t)=t^2$ in the interval $[0,1]$. By applying Proposition \ref{prop-deriv-pot} and making some basic calculus it can be seen that 
\begin{equation}
\left|\left|{^C}{^F}D^{1-\beta} g - g'\right|\right|_\infty=\max\limits_{t\in [0,1]} \left| \frac{\beta}{1-\beta}2t-\frac{2\beta}{(1-\beta)^2}\left[1-e^{-\frac{1-\beta}{\beta}t}\right] \right|
= 2\left(\frac{\beta}{1+\beta}\right)^2 \ln \beta\end{equation}
and
\begin{equation}\left|\left|{^C}D^{1-\beta} g - g'\right|\right|_\infty=\frac{2\beta}{1+\beta}\left(\frac{\Gamma(2+\beta)}{1+\beta}\right)^{\frac{1}{\beta}}.\end{equation}
Then,
$$\lim\limits_{\beta\rightarrow 0^+}\frac{\left|\left|{^C}{^F}D^{1-\beta} g - g'\right|\right|_\infty}{\left|\left|{^C}D^{1-\beta} g - g'\right|\right|_\infty}=\lim\limits_{\beta\rightarrow 0^+}\frac{\beta\ln \beta}{1+\beta}\left(\frac{1+\beta}{\Gamma(2+\beta)}\right)^{\frac{1}{\beta}}=0.$$
In this case, we can conclude that the CF derivative converges faster than the C derivative. 
\end{example}

From Examples \ref{ex1} and \ref{ex2} it can be stated that the speed of convergence depends on the functions involved instead of the fractional derivatives.   

Finally, let us compare the speed of convergence in the $L^1$ norm. Consider the function 
\begin{equation}\label{g}
\begin{array}{rcl}
g\colon[0,T] & \rightarrow &\bbR\\
t & \rightarrow &g(t)=t^m, \qquad m\,\in \, \bbN.
\end{array}
\end{equation}
From Proposition \ref{prop-deriv-pot} we have that 
\begin{equation}\label{L1-CF-g}
\begin{split}
\left|\left|{^C}{^F}D^{1-\beta} g-g'\right|\right|_{L^1(0,T)}  =\int_0^T \frac{m}{1-\be}t^{m-1}\left| \beta - \G(m)\mathcal{E
}_{1,m}\left( -\frac{1-\beta}{\beta} t \right) \right| \dd t. 
\end{split}
\end{equation}
and 
\begin{equation}\label{L1-C-g}
\begin{split}
\left|\left|{^C}D^{1-\beta} g-g'\right|\right|_{L^1(0,T)} &=\int_0^T mt^{m-1} \left| 1-t^{\be}\frac{\G(m)}{\G(m+\be)}\right| \dd t. 
\end{split}
\end{equation}

With the aim to compute the integrals \eqref{L1-CF-g} and \eqref{L1-C-g} we present the next Lemma.

  \begin{lemma}\label{inversas} Let $m\,\in \, \bbN-\{1\}$, and let $t^*\colon (0,1)\rightarrow \bbR^+$ and $s^*\colon (0,1) \rightarrow \bbR^+$ be the functions defined as
		\begin{equation}\label{t*-def}
		t^*(\beta)=v \quad \text{ if }\quad \Gamma(m)\mathcal{E
}_{1,m}\left(-\frac{1-\beta}{\beta}v\right)=\beta
		\end{equation}
		and 
	\begin{equation}\label{s*-def}
		s^*(\beta)=w \quad \text{ if }\quad w^\be\frac{\G(m)}{\G(m+\be)}=1
		\end{equation}
		respectively. Then we have that 
				\begin{enumerate}[a)] 
			\item $m-1\leq t^*(\be)$ for every $\be \in $  $(0,1)$.
			\item $m-1\leq s^*(\beta)$  for every $\be \in $ $(0,1)$.
		\end{enumerate}
			\end{lemma}

Finally, we present the next theorem:

\begin{theo}\label{comparoL1} Let $g$ be the function defined in \eqref{g}, $m \geq 2$ a fixed natural, and $T\in[0,m-1]$. Then
\begin{equation*}
\begin{split}
\lim\limits_{\beta\rightarrow 0^+}\frac{\left|\left|{^C}{^F}D^{1-\beta} g-g'\right|\right|_{L^1(0,T)}}{\left|\left|{^C}D^{1-\beta} g-g'\right|\right|_{L^1(0,T)}}&=\frac{m-T}{T}\frac{1}{\Psi(m +1)-\ln T},
\end{split}
\end{equation*}
where  $\Psi(\cdot)$ is the Psi function (or Digamma function), defined as $\Psi(x)=\frac{\Gamma'(x)}{\Gamma(x)}$ for $x\in \bbR-\mathbb{Z}^-_0$.
\end{theo}

\begin{proof}
Let $m \geq 2$ a fixed natural and $T\in[0,m-1]$. Observe that $\beta-\Gamma(m) \mathcal{E}_{1,m}\left(-\frac{1-\beta}{\beta}\cdot0\right)<0$.
Being $T\leq m-1$, by Proposition \ref{inversas} we deduce that
$$\beta -\Gamma(m)\mathcal{E}_{1,m}\left(-\frac{1-\beta}{\beta}t\right)<0.$$
Then, from \eqref{L1-CF-g} and \eqref{L1-C-g} we have
\begin{equation*}
\begin{split}
\left|\left|{^C}{^F}D^{1-\beta} g-g'\right|\right|_{L^1(0,T)}&=\frac{\Gamma(m+1)}{1-\beta}\int_0^T t^{m-1}\mathcal{E
}_{1,m}\left(-\frac{1-\beta}{\beta}t\right)dt-\frac{\beta}{1-\beta}T^m\\
&=\frac{T^m}{1-\beta}\left(\Gamma(m+1)\mathcal{E
}_{1,m+1}\left(-\frac{1-\beta}{\beta}T\right)-\beta\right).
\end{split}
\end{equation*}
Analogously, we can see that
\begin{equation*}
\begin{split}
\left|\left|{^C}D^{1-\beta} g-g'\right|\right|_{L^1(0,T)}&=\frac{T^m}{\G(m +\beta+1)}\left(\G(m +\beta+1)-\G(m+1) T^\beta\right).
\end{split}
\end{equation*}
Therefore,
\begin{equation*}
\begin{split}
\frac{\left|\left|{^C}{^F}D^{1-\beta} g-g'\right|\right|_{L^1(0,T)}}{\left|\left|{^C}D^{1-\beta} g-g'\right|\right|_{L^1(0,T)}}&=\frac{\G(m +\beta+1)}{1-\beta}\frac{\frac{\Gamma(m+1)}{\beta} \mathcal{E}_{1,m+1}\left(-\frac{1-\beta}{\beta}T\right)-1}{\frac{\G(m +\beta+1)-\G(m+1) T^\beta}{\beta}}.
\end{split}
\end{equation*}
Now, from \cite[Example 4.1]{Diethelm}
$$\mathcal{E}_{1,m+1}\left(-\frac{1-\beta}{\beta}T\right)=\frac{1}{\left(-\frac{1-\beta}{\beta}T\right)^m}\left(e^{-\frac{1-\beta}{\beta}T}-\sum\limits_{k=0}^{m-1}\frac{\left(-\frac{1-\beta}{\beta}T\right)^{k}}{k!}\right),$$
from where we conclude that
\begin{equation*}
\begin{split}
\lim\limits_{\beta\rightarrow 0^+}\frac{\Gamma(m+1)}{\beta}\mathcal{E}_{1,m+1}\left(-\frac{1-\beta}{\beta}T\right)&=\frac{m}{T}.
\end{split}
\end{equation*}
By the other side
\begin{equation*}
\begin{split}
\lim\limits_{\beta\rightarrow 0^+}\frac{\G(m +\beta+1)-\G(m+1) T^\beta}{\beta}&=\G'(m +1)-\G(m+1)\ln T.
\end{split}
\end{equation*}
Then,
\begin{equation*}
\begin{split}
\lim\limits_{\beta\rightarrow 0^+}\frac{\left|\left|{^C}{^F}D^{1-\beta} g-g'\right|\right|_{L^1(0,T)}}{\left|\left|{^C}D^{1-\beta} g-g'\right|\right|_{L^1(0,T)}}&=\frac{m-T}{T}\frac{1}{\Psi(m +1)-\ln T},
\end{split}
\end{equation*}
and the thesis holds.
\end{proof}

Proposition \ref{comparoL1} affirms that the speed of convergence for $t^m$ vary depending on the power $m$ and the interval length $T$ for the $L^1$ norm. By taking $T=1$ and $T=m-1$  it can be seen that the speed of convergence could change as it is shown in table \ref{tabla}.

\begin{table}[ht]
\caption{Different speed of convergence}
\centering
  \begin{tabular}{ | c | c | c | }
    \hline
    m & $\lim\limits_{\beta\rightarrow 0^+}\frac{\left|\left|{^C}{^F}D^{1-\beta} g-g'\right|\right|_{L^1(0,1)}}{\left|\left|{^C}D^{1-\beta} g-g'\right|\right|_{L^1(0,1)}}$ & $\lim\limits_{\beta\rightarrow 0^+}\frac{\left|\left|{^C}{^F}D^{1-\beta} g-g'\right|\right|_{L^1(0,m-1)}}{\left|\left|{^C}D^{1-\beta} g-g'\right|\right|_{L^1(0,m-1)}}$ \\ \hline
    3 & 1.592207522 &  0.8881460240\\ \hline
    4 & 1.991876242 &  0.8179851126\\ \hline
    5 & 2.344504178 &  0.7816816178\\ \hline
    6 & 2.669821563 &  0.7594559202\\ \hline
  \end{tabular}
\label{tabla}
\end{table}

\section{Conclusions}

We have analyzed the order of convergence  of different fractional differential operators to the ordinary derivative, when the order of derivation tends to one, for $L^1$ and $L^\infty$ norms. We proved that the derivatives have a similar order of convergence for both norms (in fact the order is a number $r$ in the interval $(0,1)$). By the contrast, the error estimate for the RL operator did not, in general,  converged to zero for the  $L^1$ and $L^\infty$ norms. Finally, we tried to  compared the speed of convergence for the C and  CF derivatives, concluding that, in general, neither of them is faster than the other.

\section{Acknowledgements}

\noindent The present work has been sponsored by the Projects PIP N$^\circ$ 0275 from CONICET--Univ. Austral, ANPCyT PICTO Austral 2016 N$^\circ 0090$, Austral N$^\circ 006-$INV$00020$ (Rosario, Argentina) and European Unions Horizon 2020 research and innovation programme under the Marie Sklodowska-Curie Grant Agreement N$^\circ$ 823731 CONMECH.

\end{document}